\documentclass[12pt]{amsart}
\usepackage{amsfonts,amsthm,latexsym,amsmath,amssymb,amscd,amsmath, mathrsfs, url, cite, filecontents, mathtools}

\newcounter{count}
\DeclareMathOperator{\sign}{sgn}
\numberwithin{count}{section}
\newtheorem{Lemma}[count]{Lemma}

\newtheorem{Corollary}[count]{Corollary}
\newtheorem{Theorem}[count]{Theorem}

\begin{document}

\author[T.~H.~Nguyen]{Thu Hien Nguyen}
\address{Department of Mathematics \& Computer Sciences, 
V.~N.~Karazin Kharkiv National University,
4 Svobody Sq., Kharkiv, 61022, Ukraine}
\email{nguyen.hisha@karazin.ua}

\author[A.~Vishnyakova]{Anna Vishnyakova}
\address{Department of Mathematics \& Computer Sciences, V.~N.~Karazin Kharkiv National University,
4 Svobody Sq., Kharkiv, 61022, Ukraine}
\email{anna.vishnyakova@karazin.ua}

\title[Entire functions of the Laguerre--P\'olya I class]
{On the number of real zeros of real  entire functions  with a non-decreasing 
sequence of the second quotients of Taylor coefficients}

\begin{abstract}
For an  entire function $f(z) = \sum_{k=0}^\infty  a_k z^k,$
$a_k >0,$ we define the sequence  of the second quotients of Taylor coefficients  
$Q := \left( \frac{a_k^2}{a_{k-1}a_{k+1}}  \right)_{k=1}^\infty$. We find 
new necessary conditions for a function with a non-decreasing sequence 
$Q$ to belong to the Laguerre--P\'olya class of type I. We also estimate the
possible number of nonreal zeros for a function with a non-decreasing sequence 
$Q.$
\end{abstract}



\maketitle

\section{Introduction}
The topic of zero distribution of entire functions has been the subject of study and 
discussion of mathematicians for many years (see, for example, \cite{iv}). 
In the present paper, we consider a class of entire functions with positive Taylor 
coefficients and investigate the condition for them to belong to the Laguerre--P\'olya class of type I. 
We give the definitions of the Laguerre--P\'olya class and the Laguerre--P\'olya class of type I.

{\bf Definition 1}. {\it A real entire function $f$ is said to be in the {\it
Laguerre--P\'olya class}, written $f \in \mathcal{L-P}$, if it can
be expressed in the form
\begin{equation}
\label{lpc}
 f(z) = c z^n e^{-\alpha z^2+\beta z}\prod_{k=1}^\infty
\left(1-\frac {z}{x_k} \right)e^{zx_k^{-1}},
\end{equation}
where $c, \alpha, \beta, x_k \in  \mathbb{R}$, $x_k\ne 0$,  $\alpha \ge 0$,
$n$ is a nonnegative integer and $\sum_{k=1}^\infty x_k^{-2} < \infty$. }

{\bf Definition 2}. {\it  A real entire function $f$ is said to be in the {\it Laguerre--
P\'olya class of type I}, 
written $f \in \mathcal{L-P} I$, if it can
be expressed in the following form  
\begin{equation}  \label{lpc1}
 f(z) = c z^n e^{\beta z}\prod_{k=1}^\infty
\left(1+\frac {z}{x_k} \right),
\end{equation}
where $c \in  \mathbb{R},  \beta \geq 0, x_k >0 $, 
$n$ is a nonnegative integer,  and $\sum_{k=1}^\infty x_k^{-1} <
\infty$.} As usual, the product on the right-hand sides in both definitions can be
finite or empty (in the latter case the product equals 1).

These classes are important for the theory of entire functions since 
the  hyperbolic polynomials (i.e. real polynomials with only real zeros), 
or hyperbolic polynomials  with nonnegative coefficients
converge locally  uniformly to these and only these functions. 
The following  prominent theorem 
states even a stronger fact. 

{\bf Theorem A} (E.~Laguerre and G.~P\'{o}lya, see, for example,
\cite[p. ~42--46]{HW}) and \cite[chapter VIII, \S 3]{lev}). {\it   

(i) Let $(P_n)_{n=1}^{\infty},\  P_n(0)=1, $ be a sequence
of real polynomials having only real zeros which  converges uniformly on the disc 
$|z|\leq A, A > 0.$ Then this sequence converges locally uniformly in $\mathbb{C}$ 
to an entire function 
from the $\mathcal{L-P}$ class.

(ii) For any $f \in \mathcal{L-P}$ there exists a sequence of real polynomials 
with only real zeros which converges locally uniformly to $f$.

(iii) Let $(P_n)_{n=1}^{\infty},\  P_n(0)=1, $ be a sequence
of real polynomials having only real negative zeros which  
converges uniformly on the disc $|z| \leq A, A > 0.$ Then this 
sequence converges locally uniformly in $\mathbb{C}$ to an entire function
 from the class $\mathcal{L-P}I.$
 
(iv) For any $f \in \mathcal{L-P}I$ there is a
sequence of real polynomials with only real nonpositive 
zeros which  converges locally uniformly to $f$.}

Numerous properties and features of the
Laguerre--P\'olya  class and the Laguerre--P\'olya class of type I can be found in the works 
\cite[p. 100]{pol}, \cite{polsch}  and  \cite[Kapitel II]{O} (also see the
survey \cite{iv} on the zero distribution of  entire functions, its sections and tails).
Note that for a real entire function (not identically zero) of the order less than $2$ the property of having 
only real zeros is  equivalent to belonging to the Laguerre--P\'olya class. Also, 
for a real entire function with positive coefficients of the order 
less than $1$ having only real negative zeros is  equivalent to belonging to the Laguerre--P\'olya 
class of type I. In particular, the same property is valid for polynomials.

Let  $f(z) = \sum_{k=0}^\infty a_k z^k$  be an entire function with 
real nonzero coefficients. We define the quotients $p_n$ and $q_n$:
\begin{align*}
p_n=p_n(f)&:=\frac{a_{n-1}}{a_n}, \quad n\geq 1, \\
q_n=q_n(f)&:= \frac{p_n}{p_{n-1}} = \frac{a_{n-1}^2}{a_{n-2}a_n}, \quad n\geq 2.
\end{align*}
From these definitions it follows straightforwardly that
\begin{align*}
& a_n=\frac{a_0}{p_1p_2\cdots p_n}, \quad n\geq 1, \\
& a_n = a_1\Big(\frac{a_1}{a_0} \Big)^{n-1} \frac{1}{q_2^{n-1}q_3^{n-2}
\cdots q_{n-1}^2 q_n}, \quad n\geq 2.
\end{align*}

It is rather a complicated problem to understand whether a given entire function has only real zeros. 
However, in 1926, J. I. Hutchinson found quite a simple sufficient condition for an entire function 
with positive coefficients to have only real zeros.

{\bf Theorem B} (J. ~I. ~Hutchinson, \cite{hut}). { \it Let $f(z)=
\sum_{k=0}^\infty a_k z^k$, $a_k > 0$ for all $k$. 
Then $q_n(f)\geq 4$, for all $n\geq 2,$  
if and only if the following two conditions are fulfilled:\\
(i) The zeros of $f$ are all real, simple and negative, and \\
(ii) The zeros of any polynomial $\sum_{k=m}^n a_kz^k$, $m < n,$  formed 
by taking any number 
of consecutive terms of $f $, are all real and non-positive.}
For some extensions of Hutchinson's results see,
for example, \cite[\S4]{cc1}.

A special entire function $g_a(z) =\sum _{k=0}^{\infty}a^{-k^2} z^k $, $a>1,$
known as a \textit{partial theta function} (the classical Jacobi theta function is defined 
by the series $\theta(z) := \sum_{k = - \infty}^{\infty} a^{-k^2} z^k $),  was investigated 
by many mathematicians and has an important role. Note that  $q_n(g_a)=a^2$ for all $n.$ 
The  survey \cite{warnaar} by S.~O.~Warnaar contains the history of
investigation of the partial theta function and some of its main properties. 

In particular, in the paper \cite{klv} it was explained that  for every $n\geq 2,$ there exists a constant $c_n >1$ 
such that  for each $n \in \mathbb{N},$ $S_{n}(z,g_a):=\sum _{j=0}^{n} a^{-j^2} z^j  \in \mathcal{L-P}$ if and only if
$ a^2 \geq c_n.$ The notation of the constants $c_n$ having this property will be further used.

{\bf Theorem C} (O. ~Katkova, T. ~Lobova, A. ~Vishnyakova, \cite{klv}).  {\it There exists a constant 
$q_\infty $ $(q_\infty\approx 3{.}23363666) $ such that:
\begin{enumerate}
\item
$g_a(z) \in \mathcal{L-P} \Leftrightarrow \ a^2\geq q_\infty ;$
\item
$g_a(z) \in \mathcal{L-P} \Leftrightarrow \ $  there exists $z_0 \in (- a^3, -a)$ 
such that $ \  g_a(z_0) \leq 0$
\item  
if there exists $z_0 \in (- a^3, -a)$ 
such that $ \  g_a(z_0)< 0,$ then $ a^2 > q_\infty;$
\item
for a given $n\geq 2$ we have $S_{n}(z,g_a) \in \mathcal{L-P}$ $ \  \Leftrightarrow \ $
there exists $z_n \in (- a^3, -a)$ such that $ \ S_{n}(z_n,g_a) \leq 0;$
\item
if there exists $z_n \in (- a^3, -a)$ 
such that $ \  S_{n}(z_n, g_a) < 0,$ then $ a^2 > c_n;$
\item
$ 4 = c_2 > c_4 > c_6 > \cdots $  and    $\lim_{n\to\infty} c_{2n} = q_\infty ;$
\item
$ 3= c_3 < c_5 < c_7 < \cdots $  and    $\lim_{n\to\infty} c_{2n+1} = q_\infty .$
\end{enumerate}}
Calculations show that
$c_4 = 1+ \sqrt{5} \approx 3{.}23607,$ $c_6 \approx 3{.} 23364$ and
 $c_5 \approx  3{.} 23362,$ $c_7 \approx 3{.}23364.$

The partial theta function is of interest to many areas such as statistical physics and 
combinatorics \cite{sokal}, Ramanujan type $q$-series \cite{warnaar1}, asymptotic 
analysis and the theory of (mock) modular forms, etc. There is a series of works by 
V.P. ~Kostov dedicated to various properties of zeros of 
the partial theta function and its derivative (see \cite{ kos5, kos5.0} and the 
references therein). The  paper \cite{kosshap} among the other results explains the role 
of the constant $q_\infty $  in the study of the set of entire functions with positive 
coefficients having all Taylor truncations with only real zeros. In \cite{klv1}, the following questions are investigated: whether the Taylor sections 
of the function $\prod \limits_{k=1}^\infty \left(1 + \frac{z}{a^k} \right)$, 
$a > 1,$ and $\sum_{k=0}^{\infty}\frac{z^k}{k!a^{k^2}},$ $a\geq 1,$ belong to the 
Laguerre--P\'olya class of type I. In \cite{BohVish} and 
\cite{Boh}, some important special functions with non-decreasing sequence 
of the second quotients 
of Taylor coefficients  are studied. 

The first author studied a special function related to the partial theta function and 
the Euler function 
$$f_a(z) = \sum_{k=0}^\infty \frac{z^k}{(a^k+1)(a^{k-1}+1)
\cdots (a+1)}, \quad a>1,$$  which is also known 
as the $q$-Kummer function $\prescript{}{1}{\mathbf{\phi}}_1(q;-q; q,-z)$, 
where $q=1/a$ (see \cite{GR}, formula (1.2.22)). Note that its second 
quotients of Taylor coefficients are $$q_n(f_a) = \frac{a^n + 1}{a^{n - 1} + 1}, $$  which is an increasing sequence in $n$ for $a > 1$, with the limit value given by $a$. In \cite{ngth1}, the conditions  were found for this function to 
belong to the Laguerre--P\'olya class.

It turns out that for many important entire functions with positive
coefficients $f(z)=\sum_{k=0}^\infty a_k z^k $ (for example, the partial theta function
from \cite{klv}, functions from \cite{BohVish} and \cite{Boh}, 
the $q$-Kummer function $\prescript{}{1}{\mathbf{\phi}}_1(q;-q; q,-z)$ 
and others) the following two conditions
are equivalent: 

(i) $f$ belongs to the Laguerre--P\'olya class of type I,  

and 

(ii) there exists $x_0 \in [-\frac{a_1}{a_2},0]$ such that $f(x_0) \leq 0.$

In our previous work we proved the following necessary condition for a function to belong
to the Laguerre--P\'olya class. 

{\bf Theorem D}  (T. ~H. ~Nguyen, A. ~Vishnyakova, \cite{ngthv3}).
{\it Let $f(z)=\sum_{k=0}^\infty a_k z^k$, $a_k > 0$ for all $k,$  be an 
entire function. Suppose that  
$ q_2(f) \leq q_3(f) .$  If the function $f$ belongs to the Laguerre-P\'olya class, 
then there exists $x_0 \in [- \frac{a_1}{a_2}, 0] $ such that $f(x_0) \leq 0.$}

In \cite{ngthv4} we have obtained a criterion for belonging to the 
Laguerre--P\'olya class of type I for real entire
functions with the regularly non-decreasing sequence of second quotients 
of Taylor coefficients in terms of the existence of a point $x_0$ as in Theorem D. It was previously shown in \cite{ngthv2} that if $f(z)=\sum_{k=0}^\infty a_k z^k $, $a_k > 0$ 
for all $k,$ is an entire function such that $q_2 \leq q_3 \leq 
q_4 \leq \cdots, $   and  $\lim\limits_{n\to \infty} q_n(f) = c \geq q_\infty,$
then the function $f$ belongs to the  Laguerre--P\'olya class, where $q_\infty$ is a constant from Theorem C.

In the present paper we prove that the following  conditions on the second quotients $q_k$ are necessary for 
the function to belong to the Laguerre--P\'olya I class:
\begin{Theorem}
\label{th:mthm1}
Let $f(z) = \sum_{k=0}^\infty  a_k z^k,$ $a_k >0,  k=0, 1, 2, \ldots ,$ be an entire function 
such that $q_2(f) \leq q_3(f) \leq q_4(f) \leq \cdots.$
If $f \in \mathcal{L-P} I$, then for any $k = 1, 2, 3, \ldots,$ the following inequality 
holds: $q_{2n+1} > c_{2k+1} $ ($c_{2k+1}$ defined as in Theorem C).
\end{Theorem}

\begin{Corollary}
\label{Cor1}
Let $f(x) = \sum_{k=0}^\infty  a_k x^k,$ $a_k >0,  k=0, 1, 2, \ldots ,$ be an entire function 
such that $q_2(f) \leq q_3(f) \leq q_4(f) \leq \cdots.$
If $f \in \mathcal{L-P}$, then $q_2(f) > 3.$
\end{Corollary}

In \cite{ngthv4} we obtained the following  result. 

{\bf Theorem E} (T. ~H. ~Nguyen, A. ~Vishnyakova, \cite{ngthv4}).
{\it Let $f(z) = \sum_{k=0}^\infty  a_k z^k,$
$a_k >0,  k=0, 1, 2, \ldots ,$ be an entire function 
such that $2\sqrt[3]{2} \approx 2{.}51984 \leq q_2(f) \leq q_3(f) \leq q_4(f) \leq \cdots.$ Then all but a finite number of zeros of $f$ are real and
simple.}

Our next theorem estimates the possible number of nonreal zeros
for such functions.
\begin{Theorem}
\label{th:mthm2}
Let $f(z) = \sum_{k=0}^\infty  a_k z^k,$ $a_k >0,  k=0, 1, 2, \ldots ,$ be an entire function 
such that $2\sqrt[3]{2} \approx 2{.}51984 \leq  q_2(f) \leq q_3(f) \leq q_4(f) \leq \cdots.$
If there exist $j_0 = 2, 3, 4, \ldots$ and $m_0 \in \mathbb{N},$   
such that  $q_{j_0} \geq c_{2m_0},$ 
then the number of nonreal  zeros of  $f$ does not exceed $j_0 +2m_0 - 2$ 
($c_{2k}$ defined as in Theorem C).
\end{Theorem}

\section{Proof of Theorem \ref{th:mthm1} and Corollary \ref{Cor1}}

Without loss of generality, we can assume that $a_0=a_1=1,$ since we can 
consider a function $g(x) =a_0^{-1} f (a_0 a_1^{-1}x) $  instead of 
$f(x),$ due to the fact that such rescaling of $f$ preserves its property of 
having real zeros as well as the second quotients:  $q_n(g) =q_n(f)$ 
for all $n \in \mathbb{N}.$ During the proof instead of 
$p_n(f)$ and $q_n(f)$ we use notation $p_n$ and $q_n.$ 
It is more convenient to consider a function  $$\varphi(x) = f(-x) = 1 - x + \sum_{k=2}^\infty  \frac{ (-1)^k x^k}
{q_2^{k-1} q_3^{k-2} \cdots q_{k-1}^2 q_k}$$  instead of $f.$

Theorem D states that if $\varphi$ belongs to the Laguerre--P\'olya  class then 
there exists a point $x_0 \in [0,\frac{a_1}{a_2}] = [0, q_2] $ such that $\varphi(x_0) \leq 0.$
Let us introduce some more notation. For an entire function $\varphi$, by $S_n(x, \varphi)$ 
and $R_{n}(x, \varphi)$ we denote the $n$th partial sum and the $n$th remainder of the series,
i.e. 
$$S_n(x, \varphi) = \sum_{k=0}^n \frac{(-1)^k x^k}{q_2^{k-1} q_3^{k-2} \cdots q_{k-1}^2 q_k},$$
and 
$$R_{n}(x, \varphi) = \sum_{k=n}^\infty \frac{(-1)^k  x^k}{q_2^{k-1} q_3^{k-2} \cdots q_{k-1}^2 q_k}.$$

First, we need the following Lemma.

\begin{Lemma}
\label{th:lm1}
Let $\varphi(x) =  1 - x + \sum_{k=2}^\infty  \frac{ (-1)^k x^k}
{q_2^{k-1} q_3^{k-2} \cdots q_{k-1}^2 q_k} $  be an entire function.
Suppose that  $q_k$ are non-decreasing in $k:$    $1 < q_2 \leq q_3 \leq q_4 \leq \cdots.$
 If there exists $x_0 \in [0, q_2]$ such that  $\varphi(x_0) \leq 0,$ then $x_0 \in (1, q_2].$
\end{Lemma}

\begin{proof}  For $x \in [0,1]$ we have: 

$$1 \geq x  > \frac{x^2}{q_2} >  \frac{x^3}{q_2^2 q_3}  > 
\frac{x^4}{q_2^3 q_3^2 q_4 } > \cdots, 
$$

whence 
\begin{equation}
\label{mthm1.1}
 \varphi(x) >0 \quad \mbox{for all}\quad x\in [0,1].
\end{equation}

\end{proof}

\begin{Lemma}
\label{th:lm2}
Let $\varphi(x) =  1 - x + \sum_{k=2}^\infty  \frac{ (-1)^k x^k}
{q_2^{k-1} q_3^{k-2} \cdots q_{k-1}^2 q_k} $  be an entire function.
Suppose that  $q_k$ are non-decreasing in $k:$ $1 < q_2 \leq q_3 \leq q_4 \leq \cdots.$
 If there exists $x_0 \in (1, q_2]$ such that  $\varphi(x_0)\leq0,$ then for any 
 $n \in \mathbb{N},$ $S_{2n+1}(x_0) < 0.$
\end{Lemma}

\begin{proof}
Suppose that   $x \in (1, q_2].$  Then we obtain 
\begin{equation}
\label{mthm1.2}  1 <  x  \geq  \frac{x^2}{q_2} > \frac{x^3}{q_2^2 q_3} > 
\cdots >  \frac{  x^k}{q_2^{k-1} q_3^{k-2} \cdots q_{k-1}^2 q_k} > \cdots 
\end{equation}
For an arbitrary  $n \in {\mathbb{N}}$  we have:
\begin{align*}
\varphi(x) = S_{2n+1}(x, \varphi) + R_{2n+2}(x, \varphi).
\end{align*}
By (\ref{mthm1.2}) and the Leibniz criterion for alternating series, we conclude 
that $R_{2n+2}(x, \varphi) >0$ for all   $x \in (1, q_2],$
or
\begin{equation}
\label{mthm1.3}   \varphi(x) > S_{2n+1}(x, \varphi)\quad \mbox{for all} \quad 
x \in (1, q_2], n \in {\mathbb{N}}.
\end{equation} 
Consequently, if there exists a point $x_0 \in (1, q_2]$ such that $\varphi(x_0) \leq 0,$ 
then for any $n \in \mathbb{N}$ we have $S_{2n+1}(x_0) < 0.$
\end{proof}

Thus, we proved that if $\varphi \in \mathcal{L-P},$ then there exists  $x_0 \in (1, q_2]$ 
such that the inequalities $S_{2n+1}(x_0) < 0$ hold for any $n \in \mathbb{N}.$

In \cite{ngthv2} it was proved that if an entire function $\varphi(x)  = 1 - x + \sum_{k=2}^\infty  
\frac{ (-1)^k x^k}{q_2^{k-1} q_3^{k-2} \cdots q_{k-1}^2 q_k} $ belongs to the 
Laguerre--P\'olya  class, where $0 < q_2 \leq q_3 \leq q_4 \leq \cdots,$ then $q_2 \geq 3$  
(see\cite[Lemma 2.1]{ngthv2}). 
So  we assume that $q_2 \geq 3.$

\begin{Lemma}
\label{th:lm3}
Let $\varphi(x) =  1 - x + \sum_{k=2}^\infty  \frac{ (-1)^k x^k}
{q_2^{k-1} q_3^{k-2} \cdots q_{k-1}^2 q_k} $  be an entire function.
Suppose that   $ 3 \leq q_2 \leq q_3 \leq q_4 \cdots.$ Then
the inequality $S_{2n+1}(x, \varphi) \geq S_{2n+1}( \sqrt{q_{2n+1}} x, g_{\sqrt{q_{2n+1}}})$ 
holds for any $n \in \mathbb{N}$ and any $x \in (1, q_2]$ (here $g_a$ is the partial theta function 
and $S_{2n+1}(y, g_a)$ is its $(2n+1)$-th partial sum
at the point  $y$).
\end{Lemma}

\begin{proof}
We have
\begin{align}
\label{m5}
&S_{2n+1}(x, \varphi) = (1 - x) + \left(\frac{x^2}{q_2} - \frac{x^3}{q_2^2 q_3}\right)
+ \left(\frac{x^4}{q_2^3 q_3^2 q_4} - \frac{x^5}{q_2^4 q_3^3 q_4^2 q_5}\right)\\
&\nonumber + \cdots+ 
\left(\frac{x^{2n}}{q_2^{2n-1} q_3^{2n-2} \cdots 
q_{2n-1}^2 q_{2n}} - \frac{x^{2n+1}}{q_2^{2n} q_3^{2n-1} \cdots 
q_{2n}^2 q_{2n+1}}\right).
\end{align}
Under our assumptions,  $q_k$ 
are non-decreasing in $k.$  We prove that for any fixed $k = 1, 2, \ldots, n$ and 
$x \in (1, q_2],$  the following inequality holds:

\begin{align*}
& \frac{x^{2k}}{q_2^{2k-1} q_3^{2k-2}\cdots q_{2k-1}^2q_{2k}} - \frac{x^{2k+1}}{q_2^{2k} q_3^{2k-1}\cdots q_{2k-1}^3 q_{2k}^2 q_{2k+1}}  \\
& \quad \geq  \frac{x^{2k}}{q_{2k+1}^{2k-1}q_{2k+1}^{2k-2}\cdots q_{2k+1}^2 q_{2k+1}} - 
 \frac{x^{2k+1}}{q_{2k+1}^{2k} q_{2k+1}^{2k-1}\cdots q_{2k+1}^2q_{2k+1}} \\
&\qquad = \frac{x^{2k}}{q_{2k+1}^{k(2k-1)}} - \frac{x^{2k+1}}{q_{2k+1}^{k(2k+1)}} = 
\frac{x^{2k}}{q_{2k+1}^{k(2k-1)}}\cdot \left( 1 - \frac{x}{q_{2k+1}^{2k}}\right).
\end{align*}
For  $x \in (1, q_2]$ and any fixed $k=1, 2, \ldots , n,$ we define the following function:
\begin{multline*}
F(q_2, q_3, \ldots, q_{2k}, q_{2k+1}):=\frac{x^{2k}}{q_2^{2k-1} q_3^{2k-2}\cdots q_{2k-1}^2 q_{2k}} \\\ - \frac{x^{2k+1}}{q_2^{2k} q_3^{2k-1}\cdots q_{2k-1}^3 q_{2k}^2 
q_{2k+1}}.
\end{multline*}
We can observe that
\begin{multline*}
\frac{\partial F(q_2, q_3, \ldots, q_{2k}, q_{2k+1}) }
 {\partial q_2} = - \frac{(2k-1) \cdot x^{2k}}{q_2^{2k} q_3^{2k-2} \cdots q_{2k-1}^2 q_{2k}} \\+ 
 \frac{2k\cdot x^{2k+1}}{q_2^{2k+1}q_3^{2k-1} \cdots q_{2k-1}^3 q_{2k}^2q_{2k+1}} < 0
 \Leftrightarrow x < 
 \left(1 - \frac{1}{2k} \right) \cdot q_2 q_3 \ldots q_{2k}q_{2k+1}.
 \end{multline*}
Therefore, since    $\left(1 - \frac{1}{2k} \right) q_2 q_3 \cdots q_{2k}q_{2k+1} 
\geq \frac{1}{2}  q_2 q_3 \cdots q_{2k}q_{2k+1}  \geq \frac{1}{2}q_2q_3 > q_2$ (under our assumptions $q_3 \geq q_2 \geq 3$), 
we conclude that the function $F(q_2, q_3, \ldots, q_{2k}, q_{2k+1})$ 
is decreasing in $q_2$ for each fixed $x \in (1, q_2]$.  Since $q_2 \leq q_3,$  for $k=1$ we get:
$$ F(q_2, q_3) =  \frac{x^2}{q_2} 
- \frac{x^3}{q_2^2 q_3} \geq \frac{x^2}{q_3} 
- \frac{x^3}{q_3^2 q_3} = \frac{x^2}{q_3} - \frac{x^3}{q_3^3}, $$
and the desired inequality is proved for $k=1.$ For $k\geq 2$ we have:
\begin{align*}
&F(q_2, q_3, q_4, \ldots, q_{2k}, q_{2k+1}) \geq F(q_3, q_3, q_4,  \ldots, q_{2k}, q_{2k+1})\\
&\quad =\frac{x^{2k}}{q_3^{4k-3} q_4^{2k-3} \cdots q_{2k-1}^2 q_{2k}} - 
\frac{x^{2k+1}}{q_3^{4k-1}q_4^{2k-2} \cdots q_{2k-1}^3 q_{2k}^2 q_{2k+1}}.
\end{align*}
Further, we consider its derivative with respect to  $q_3$: 
\begin{multline*}
\frac{\partial F(q_3, q_3, q_4,  \ldots, q_{2k}, q_{2k+1})}
 {\partial q_3} = - \frac{(4k-3)\cdot x^{2k}}{q_3^{4k-2} q_4^{2k-3} \cdots q_{2k-1}^2 q _{2k}} 
 \\+ \frac{(4k-1) \cdot x^{2k+1}}{q_3^{4k}q_4^{2k-2} \cdots q_{2k+1}} < 0 
 \Leftrightarrow x < \frac{4k-3}{4k-1}q_3^2 q_4 \ldots q_{2k-1}^3 q_{2k}^2 q_{2k+1}. 
 \end{multline*}
Under our assumptions, $$\frac{4k-3}{4k-1}\cdot q_3^2 q_4 \ldots q_{2k+1}\geq \frac{5}{7}\cdot q_3^2 q_4q_{5}  
> q_2,$$ we obtain that $F(q_3, q_3, q_4, \ldots, q_{2k}, q_{2k+1})$ 
is decreasing in $q_3$ for each fixed $x \in (1, q_2]$ and, since $q_3 \leq q_4,$
we receive:
 $$F(q_3, q_3, q_4 \ldots, q_{2k}, q_{2k+1}) \geq F(q_4, q_4, q_4, q_5,  \ldots, q_{2k}, q_{2k+1}).$$
Thus, for the $l$th step we have:
 \begin{align*}
 &F(q_{l-1}, q_{l-1}, \ldots, q_{l-1}, q_l, q_{l+1}, \ldots, q_{2k}, q_{2k+1}) \\
 &\quad = \frac{x^{2k}}{q_{l-1}^{(4k-l+1)(l-2)/2}q_l^{2k - l +1} q_{l+1}^{2k-l} \cdots q_{2k-1}^2q_{2k}} \\ 
 &\qquad - \frac{x^{2k+1}}{q_{l-1}^{(4k-l+3)(l-2)/2} q_l^{2k - l + 2} q_{l+1}^{2k - l + 1} \cdots q_{2k-1}^3 q_{2k}^2 q_{2k+1}}.
 \end{align*}
 We consider its partial derivative with respect to $q_{l-1}:$
 \begin{align*}
 &\frac{\partial F(q_{l-1}, q_{l-1}, \ldots, q_{l-1}, q_l, q_{l+1}, \ldots, q_{2k}, q_{2k+1})}{\partial q_{l-1}} \\
 &\quad = - \frac{\frac{1}{2}(4k-l+1)(l-2) \cdot x^{2k}}{q_{l-1}^{1+ (4k-l+1)(l-2)/2}q_l^{2k - l +1} q_{l+1}^{2k-l} \cdots q_{2k-1}^2q_{2k}} \\
 &\qquad + \frac{\frac{1}{2}(4k-l+3)(l-2) \cdot x^{2k+1}}{q_{l-1}^{1+ (4k-l+3)(l-2)/2}q_l^{2k - l + 2} q_{l+1}^{2k - l + 1} \cdots q_{2k-1}^3 q_{2k}^2 q_{2k+1}} < 0,
 \end{align*}
 which is equivalent to the inequality:
 \begin{align*}
 x < \frac{4k - l + 1}{4k - l + 3} \cdot q_{l-1}^{l-2} q_l q_{l+1} \cdots q_{2k-1} q_{2k}q_{2k + 1}. 
 \end{align*}
 The inequality above is valid, since 
 \begin{align*}
 &\frac{4k - l + 1}{4k - l + 3} \cdot q_{l-1}^{l-2} q_l q_{l+1} \cdots q_{2k-1} q_{2k}q_{2k + 1}\\
 &\quad \geq \frac{9 - l}{11 - l} \cdot q_{l-1}^{l-2} q_l q_{l+1} \cdots q_{2k-1} q_{2k}q_{2k + 1} > q_2.
 \end{align*}
 Hence, the function $F(q_{l-1}, q_{l-1}, \ldots, q_{l-1}, q_l, q_{l+1}, \ldots, q_{2k}, q_{2k+1})$ is decreasing in $q_{l-1}.$ Since, under our assumptions, $q_{l-1} \leq q_l,$ we obtain:
 \begin{multline*}
 F(q_{l-1}, q_{l-1}, \ldots, q_{l-1}, q_l, q_{l+1}, \ldots, q_{2k}, q_{2k+1}) \\ \geq
 F(q_l, q_l, \ldots, q_l, q_{l+1}, \ldots, q_{2k}, q_{2k+1}).
 \end{multline*}
 Analogously, by the same computation, at the $(2k+1)$-th step we get:
\begin{align*}
F(q_{2k}, q_{2k} \ldots, q_{2k}, q_{2k+1}) = \frac{x^{2k}}{q_{2k}^{k(2k-1)}} - \frac{x^{2k+1}}{q_{2k}^{(k+1)(2k-1)} \cdot q_{2k+1}}.
\end{align*}
Its derivative with respect to $q_{2k}$ is:
\begin{align*}
&\frac{\partial F(q_{2k}, q_{2k} \ldots, q_{2k}, q_{2k+1})}{\partial q_{2k}} = -\frac{k(2k-1)\cdot x^{2k}}{q_{2k}^{2k^2 -k+1}} \\
&\quad + \frac{(2k^2 + k - 1) \cdot x^{2k+1}}{q_{2k}^{2k^2+k}q_{2k+1}} < 0
 \Leftrightarrow x < \frac{2k^2 - k}{2k^2 + k -1} \cdot q_{2k}^{2k-1} q_{2k +1}.
\end{align*}
Since we assume that 
$$\frac{2k^2 - k}{2k^2 + k -1} \cdot q_{2k}^{2k-1} q_{2k +1} \geq \frac{2}{3} \cdot q_{2k}^{2k-1} q_{2k +1} > q_2,$$
we conclude that the function $F(q_{2k}, q_{2k} \ldots, q_{2k}, q_{2k+1})$ is decreasing in $q_{2k}$. While $q_{2k} \leq q_{2k+1},$ we get:
$$F(q_{2k}, q_{2k} \ldots, q_{2k}, q_{2k+1}) \geq F(q_{2k+1}, q_{2k+1}, \ldots,  q_{2k+1}, q_{2k+1}).$$
Thus, we obtain the following chain of inequalities:
\begin{align*}
&F(q_2, q_3, q_4, \ldots, q_{2k}, q_{2k+1}) \geq F(q_3, q_3, q_4,  \ldots, q_{2k}, q_{2k+1}) \\ 
& \quad \geq F(q_4, q_4, q_4, q_5,  \ldots, q_{2k}, q_{2k+1}) \geq \cdots  
 \geq F(q_{2k}, q_{2k}, \ldots,  q_{2k}, q_{2k+1})\\
& \quad \geq F(q_{2k+1}, q_{2k+1}, \ldots,  q_{2k+1}, q_{2k+1}).
 \end{align*}
Consequently, 
\begin{align*}
&F(q_2, q_3, q_4, \ldots, q_{2k}, q_{2k+1}) \geq F(q_{2k+1}, q_{2k+1}, \ldots,  q_{2k+1}, q_{2k+1})\\
& \quad = \frac{x^{2k}}{q_{2k+1}^{k(2k-1)}} - \frac{x^{2k+1}}{q_{2k+1}^{k(2k+1)}}.
\end{align*} 
Finally, we note that under our assumptions,  the expression $\frac{x^{2k}}{q_{2k+1}^{k(2k-1)}} - 
\frac{x^{2k+1}}{q_{2k+1}^{k(2k+1)}}$ is decreasing in $q_{2k+1}$ for each fixed $x \in (1, q_2],$
so we obtain
\begin{align*}
F(q_2, q_3, q_4, \ldots, q_{2k}, q_{2k+1}) 
 \geq \frac{x^{2k}}{q_{2k+1}^{k(2k-1)}} - \frac{x^{2k+1}}{q_{2k+1}^{k(2k+1)}} \geq 
 \frac{x^{2k}}{q_{2n+1}^{k(2k-1)}} - \frac{x^{2k+1}}{q_{2n+1}^{k(2k+1)}}.
\end{align*} 
Substituting the last inequality in (\ref{m5}) for every  $x \in (1, q_2]$ and $k=1, 2, \ldots , n,$ we get:
\begin{align}
\label{mthm1.4}
&S_{2n+1}(x, \varphi) \geq (1 - x) + \left(\frac{x^2}{q_{2n+1}} 
- \frac{x^3}{q_{2n+1}^3}\right)
+ \left(\frac{x^4}{q_{2n+1}^6} - \frac{x^5}{q_{2n+1}^{10}}\right)+ \\
&   \nonumber \cdots+ \left(\frac{x^{2n}}{q_{2n+1}^{n(2n-1)}} - \frac{x^{2n+1}}{q_{2n+1}^{n(2n+1)} }\right) 
= \sum_{k=0}^{2n+1} \frac{(-1)^kx^k}{\sqrt{q_{2n+1}}^{k(k-1)}} \\
& \nonumber = S_{2n+1}( - \sqrt{q_{2n+1}} x, g_{\sqrt{q_{2n+1}}}),
\end{align}
where $g_a$ is the partial theta function and $S_{2n+1}(y, g_a)$ is its $(2n+1)$-th partial sum
at the point  $y$. 
\end{proof}

Since we have $S_{2n+1}(x, \varphi) \geq S_{2n+1}(  - \sqrt{q_{2n+1}} x, g_{\sqrt{q_{2n+1}}})$ for 
any $n \in \mathbb{N},$  if there exists a point $x_0 \in (1, q_2]$ such that $S_{2n+1}(x_0, \varphi) \leq 0,$ then 
$S_{2n+1}( - \sqrt{q_{2n+1}} x_0, g_{\sqrt{q_{2n+1}}})< 0.$ Therefore for  $y_0 =  \sqrt{q_{2n+1}} x_0,$ 
we have  $\sqrt{q_{2n+1}} \leq y_0 \leq \sqrt{q_{2n+1}} q_2 \leq (\sqrt{q_{2n+1}})^3.$ 
Using the statement (5) of Theorem C, we obtain that  $q_{2n+1} > c_{2n+1},$ which completes 
the proof of Theorem \ref{th:mthm1}.

{\it Proof of Corollary \ref{Cor1}.} As we have proved in the previous theorem, if 
$f \in \mathcal{L-P}$, then $q_3(f) > 3.$ In \cite{ngthv3} it is proved that,
under the assumptions of the Corollary, if $q_2(f)  < 4, $
then 
$$q_3(f) \leq \frac{-q_2(f)(2q_2(f) -9)+ 2(q_2(f) -3)\sqrt{q_2(f)(q_2(f)-3)}}{q_2(f)(4 - q_2(f))}$$
(see \cite[Theorem 1.4]{ngthv3}).
We have mentioned that if $f \in \mathcal{L-P}$, then $q_2(f) \geq 3.$ If $q_2(f) = 3,$
then the inequality above states $q_3(f) \leq  3.$ This contradiction proves the Corollary \ref{Cor1}.
$\Box$
\

\section{Proof of Theorem \ref{th:mthm2}}

As in the proof of Theorem \ref{th:mthm1} we assume that $a_0=a_1=1,$
and we consider the function $\varphi(x) = f(-x) =  1 - x + \sum_{k=2}^\infty  \frac{ (-1)^k x^k}
{q_2^{k-1} q_3^{k-2} \cdots q_{k-1}^2 q_k}$ instead of $f.$ We need the following lemma.

\begin{Lemma}
\label{th:lm4}
Let $\varphi(x) =  1 - x + \sum_{k=2}^\infty  \frac{ (-1)^k x^k}
{q_2^{k-1} q_3^{k-2} \cdots q_{k-1}^2 q_k} $  be an entire function.
Suppose that    $1 <  q_2 \leq q_3 \leq q_4 \leq \cdots.$
If there exist $j_0 = 3, 4, \ldots$ and $m_0 \in \mathbb{N},$   
such that  $q_{j_0} \geq c_{2m_0},$  then
for all  $j \geq j_0 +2 m_0 -3,$ there exists $ x_j  \in (q_2q_3 \cdots  q_j, q_2q_3 \cdots  q_j q_{j+1})$ 
such that  the following inequality holds: 
$$  (-1)^j \varphi(x_j) \geq 0.$$ 
\end{Lemma}
The proof of this lemma is similar to the one of \cite[Lemma 2.1]{ngthv1}.

\begin{proof}   Choose an arbitrary $j \geq j_0 +2m_0 - 3$ and fix this $j.$ For every 
$x \in (q_2q_3 \cdots  q_j, q_2q_3 \cdots  q_j q_{j+1})$ we have
$$ 1 < x < \frac{x^2}{q_2}< \frac{x^3}{q_2^2 q_3}  < \cdots < \frac{  x^j}
{q_2^{j-1} q_3^{j-2} \cdots q_{j-1}^2 q_j} ,$$
and
$$ \frac{  x^j}
{q_2^{j-1} q_3^{j-2} \cdots q_{j-1}^2 q_j} > \frac{  x^{j+1}}
{q_2^{j} q_3^{j-1} \cdots q_{j-1}^3 q_j^2 q_{j+1}} $$
$$ > \frac{  x^{j+2}}
{q_2^{j+1} q_3^{j} \cdots q_{j-1}^4 q_j^3 q_{j+1}^2 q_{j+2}} > \cdots .$$
We observe that 
$$ (-1)^j \varphi(x)  =  \sum_{k=0}^{j-2m_0 } \frac{ (-1)^{k+j} x^k}
{q_2^{k-1} q_3^{k-2} \cdots q_{k-1}^2 q_k} + 
\sum_{k=j - 2m_0 +1}^{j+1} \frac{ (-1)^{k+j} x^k}
{q_2^{k-1} q_3^{k-2} \cdots q_{k-1}^2 q_k}  $$
$$+ \sum_{k=j+2 }^{\infty} \frac{ (-1)^{k+j} x^k}
{q_2^{k-1} q_3^{k-2} \cdots q_{k-1}^2 q_k} =: \Sigma_{1}(x) + h(x) + \Sigma_{2}(x). $$

Summands in $\Sigma_{1}(x)$ are increasing in modulus and the sign of the last (biggest)
summand is positive. So, for all $x \in (q_2q_3 \cdots  q_j, q_2q_3 \cdots  q_j q_{j+1}),$ we have
$\Sigma_{1}(x) >0.$  Summands in $\Sigma_{2}(x)$ are decreasing in modulus and the sign of the 
first (biggest) summand is positive. Consequently, for all $x \in (q_2q_3 \cdots  q_j, q_2q_3 \cdots  q_j q_{j+1}),$ we get
$\Sigma_{2}(x) >0.$ Thus, we obtain
\begin{eqnarray}
\label{e1} &   (-1)^j \varphi(x) > h(x) = \sum_{k=j-2m_0 +1 }^{j +1} \frac{ (-1)^{k+j} x^k}
{q_2^{k-1} q_3^{k-2} \ldots q_{k-1}^2 q_k}  \\
\nonumber & = - \frac{x^{j+1}}{q_2^{j} q_3^{j-1 } \ldots q_{j}^2 q_{j+1}}   +
 \frac{x^{j}}{q_2^{j-1} q_3^{j-2 } \ldots q_{j-1}^2 q_{j}} - 
  \frac{x^{j-1}}{q_2^{j-2} q_3^{j-3 } \ldots q_{j-2}^2 q_{j-1}}    \\
\nonumber &  +  \ldots + \frac{x^{j-2m_0 +2}}{q_2^{j-2m_0 + 1 } q_3^{j-2m_0} \ldots q_{j-2m_0 +1 }^2 q_{j-2m_0 +2}}  
- \frac{x^{j-2m_0 + 1}}{q_2^{j-2m_0} q_3^{j-2m_0 -1 } \ldots q_{j-2m_0 }^2 q_{j-2m_0 + 1 }} 
\end{eqnarray} 
(we rewrite the sum from the end to the beginning). After factoring out the term $\frac{  x^{j +1}}
{q_2^{j} q_3^{j-1} \ldots q_{j}^2 q_{j +1}},$ we get
\begin{eqnarray}
\label{ee1} & (-1)^j \varphi(x) > h(x) = \frac{  x^{j +1}}
{q_2^{j} q_3^{j-1} \cdots q_{j}^2 q_{j +1}} \cdot \left(- 1 + \right.   \frac{q_2 q_3 \cdots q_{j} q_{j +1}}{x} 
\\
\nonumber  &
-  \frac{(q_2 q_3 \cdots q_{j} q_{j + 1})^2}{x^2 q_{j +1}}    + 
 \frac{(q_2 q_3 \cdots q_{j} q_{j +1})^3}{x^3 q_{j+1}^2 q_{j}} -
\ldots   \\
\nonumber  &    + \frac{(q_2 q_3 \cdots q_{j} 
q_{j + 1})^{2m_0 -1}}{x^{2m_0 -1} q_{j+1}^{2m_0 -2} q_{j}^{2m_0 -3}
\cdots  q_{j-2m_0+5}^2 q_{j-2m_0+4}}\\
\nonumber  &  \left.  -   \frac{(q_2 q_3 
\cdots q_{j} q_{j+1})^{2m_0 }}{x^{2m_0 } q_{j+1}^{2m_0 -1} q_{j }^{2m_0 -2}
\cdots  q_{j-2m_0+5}^3 q_{j-2m_0+4}^2 q_{j-2m_0+3} }   \right)\\
\nonumber & =: \frac{x^{j +1}}
{q_2^{j} q_3^{j-1} \cdots q_{j}^2 q_{j +1}} \cdot \psi(x).   
\end{eqnarray}

Now we introduce some more notation. Set $y := \frac{q_2 q_3 \ldots q_{j} q_{j +1}}{x}, $
and observe that  $x \in (q_2q_3 \cdots  q_j, q_2q_3 \cdots  q_j q_{j+1})  \Leftrightarrow
 y \in (1, q_{j+1}).$ Further we change the numeration of the second quotients:
$$s_2 := q_{j+1}, \  s_3 := q_j, \  s_4 :=  q_{j-1}, \   \ldots , s_{2m_0 -1} := q_{j-2m_0 +4}, \   
s_{2m_0 } := q_{j-2m_0 +3}.  $$  
By our assumptions, $q_2 \leq q_3 \leq q_4 \leq \cdots,$ thus, we get
$s_2 \geq s_3 \geq s_4 \geq \cdots \geq s_{2m_0 } > 1,$ and $y \in (1, s_2).$ In new notation we have

\begin{equation}
\label{e2} \psi(y) =  - 1 + y - \sum_{k=2}^{2m_0}\frac{ (-1)^k y^k}
{s_2^{k-1} s_3^{k-2} \cdots s_{k-1}^2 s_k} .    
\end {equation}
We want to prove that there exists a point $ y_j \in (1, q_{j+1}) = (1, s_2)$ such that
$h(y_j) \geq 0.$ To do this we compare the expression in brackets with the
corresponding partial sum of the partial theta function. We have 
\begin{eqnarray}
\label{e33}
& \psi(y) = (- 1 + y) \\
\nonumber &  + \left(- \frac{y^2}{s_2} + \frac{y^3}{s_2^2 s_3} \right) 
 +\left(- \frac{y^4}{s_2^3 s_3^2 s_4} + \frac{y^5}{s_2^4 s_3^3 s_4^2 s_5} \right)\\
\nonumber &  + \cdots + \left(- \frac{  y^{2m_0 -2}}
{s_2^{2m_0 -3} s_3^{2m_0-4} \cdots s_{2m_0 -3}^2 s_{2m_0 -2}}  + 
\frac{  y^{2m_0 -1}}
{s_2^{2m_0 -2} s_3^{2m_0-3} \cdots s_{2m_0 -2}^2 s_{2m_0 -1}}\right)  \\
\nonumber &  - \frac{  y^{2m_0}}
{s_2^{2m_0 -1} s_3^{2m_0-2} \cdots s_{2m_0 -2}^3 s_{2m_0 -1}^2 s_{2m_0}}.  
\end{eqnarray} 
We provide estimations similar to those in the proof of  Lemma \ref{th:lm3}.
Firstly, under our assumptions, one can see that
\begin{multline}
\label{e4} - \frac{  y^{2m_0}}
{s_2^{2m_0 -1} s_3^{2m_0-2} \cdots  s_{2m_0 -1}^2 s_{2m_0}}\\ \geq
- \frac{  y^{2m_0}}
{s_{2 m_0}^{2m_0 -1} s_{2 m_0}^{2m_0-2} \cdots  s_{2 m_0}^2 s_{2 m_0}} = - \frac{y^{2m_0}}{s_{2m_0}^{m_0(2m_0 - 1)}}. 
\end{multline}
We prove that for any fixed $k = 1, 2, \ldots, m_0-1,$ the following inequality holds:
\begin{align}
\label{mm6} 
&-\frac{y^{2k}}{s_2^{2k-1}s_3^{2k-2}\cdots s_{2k}} +
 \frac{y^{2k+1}}{s_2^{2k}s_3^{2k-1}\cdots s_{2k}^2 
 s_{2k+1}} \\ 
 &\nonumber \geq -\frac{y^{2k}}{s_{2 m_0}^{2k-1}s_{2 m_0}^{2k-2}\cdots s_{2 m_0}} 
+ \frac{y^{2k+1}}{s_{2 m_0}^{2k} s_{2 m_0}^{2k-1}\cdots s_{2 m_0}^2 s_{2 m_0}} \\
& \nonumber \quad = -\frac{y^{2k}}{s_{2m_0}^{k(2k-1)}} + \frac{y^{2k+1}}{s_{2m_0}^{k(2k+1)}}.
\end{align}
Firstly, we consider (\ref{mm6}) for $k = 1.$  Since $s_2 \geq s_3$, we have
$$- \frac{y^2}{s_2} + \frac{y^3}{s_2^2 s_3} 
\geq  - \frac{y^2}{s_2} + \frac{y^3}{s_2^3}.$$
We observe that
$$\frac{\partial }{\partial s_2}\left(  - \frac{y^2}{s_2} + 
\frac{y^3}{s_2^3} \right) = 
\frac{y^2}{s_2^2} - \frac {3y^3}{s_2^4} > 0  \Leftrightarrow y  
< \frac{s_2^2}{3}.$$
The inequality above is valid since $y < q_{j+1} = s_2,$ and we suppose that if there exist $j_0 = 2, 3, 4, \ldots$ and $m_0 \in \mathbb{N},$ such that $q_{j_0} \geq c_{2m_0},$ we fix an arbitrary $j \geq j_0 + 2m_0 - 3$ and get  
$s_2 \geq s_{2m_0} =q_{j -2m_0 +3}  \geq  q_{j_0}\geq c_{2m_0}>3.$ Therefore, the function  $\left(  - \frac{y^2}{s_2} + \frac{y^3}{s_2^3} \right)$ is 
increasing in $s_2,$ whence 
\begin{equation}
\label{mm7}
- \frac{y^2}{s_2} + \frac{y^3}{s_2^2 s_3} 
\geq - \frac{y^2}{s_2} + \frac{y^3}{s_2^3} \geq - 
\frac{y^2}{s_{2 m_0}} + \frac{y^3}{s_{2 m_0}^3}.
\end{equation} 
We apply analogous reasoning to prove (\ref{mm6}) for every $k = 1, 2, \ldots, m_0-1$. 
Let us define the following function:
\begin{multline*}
H (s_2, s_3, \ldots , s_{2k}, s_{2k+1}) :=
 -\frac{y^{2k}}{s_2^{2k-1} s_3^{2k-2}\cdots s_{2k-1}^2 s_{2k}} \\+
 \frac{y^{2k+1}}{s_2^{2k} s_3^{2k-1}\cdots s_{2k-1}^3 s_{2k}^2 
 s_{2k+1}}
 \end{multline*}
 for $s_2 \geq s_3 \geq \cdots \geq s_{2k+1}.$
Obviously,
\begin{align*}
&H(s_2, s_3, \ldots , s_{2k}, s_{2k+1})  \geq H(s_2, s_3, \ldots , s_{2k}, s_{2k}) \\
& \quad = -\frac{y^{2k}}{s_2^{2k-1} s_3^{2k-2}\cdots s_{2k-1}^2 s_{2k}} +
 \frac{y^{2k+1}}{s_2^{2k} s_3^{2k-1}\cdots s_{2k-1}^3s_{2k}^3}.
 \end{align*}
We have
$$\frac{\partial H(s_2, s_3, \ldots , s_{2k}, 
s_{2k}) }{\partial s_{2k}} = \frac{y^{2k}}{s_2^{2k-1}
s_3^{2k-2}\cdots s_{2k-1}^2 s_{2k}^2} - 
\frac{3y^{2k+1}}{s_2^{2k} s_3^{2k-1}\cdots
s_{2k-1}^3 s_{2k}^4}.$$
Thus,
$$  \frac{\partial H(s_2, s_3, \ldots , s_{2k}, 
s_{2k}) }{\partial s_{2k}}  > 0 
\Leftrightarrow y  < \frac{s_2  s_3 \cdots
s_{2k-1}s_{2k}^2}{3}.$$
Since $y \in (1, s_2) \Leftrightarrow y < s_2,$  we obtain that the function 
$H(s_2, s_3, \ldots , s_{2k}, 
s_{2k})$ is increasing in $s_{2k},$
whence 
\begin{align*}
& H(s_2, s_3, \ldots, s_{2k-1}, s_{2k}, s_{2k+1})  \geq 
H(s_2, s_3, \ldots, s_{2k-1}, s_{2k}, s_{2k}) \\
& \geq  H(s_2, s_3, \ldots, s_{2k-1}, s_{2 m_0}, s_{2 m_0}) = -\frac{y^{2k}}{s_2^{2k-1} s_3^{2k-2}\cdots s_{2k-1}^2 s_{2m_0}} \\
& \quad + \frac{y^{2k+1}}{s_2^{2k} s_3^{2k-1}\cdots s_{2k-1}^3s_{2m_0}^3}.
\end{align*}
Now we consider the derivative of the latter function:
\begin{multline*}
\frac{\partial H(s_2, s_3, \ldots , s_{2k-1}, s_{2 m_0}, s_{2 m_0}) }{\partial 
s_{2k-1}} \\ = \frac{2y^{2k}}{s_2^{2k-1}s_3^{2k-2}\cdots s_{2k-1}^3 s_{2 m_0}} 
 - \frac{3y^{2k+1}}{s_2^{2k}
s_3^{2k-1}\cdots s_{2k-1}^4 s_{2 m_0}^3}.
\end{multline*}
Hence,
$$ \frac{\partial H(s_2, s_3, \ldots , s_{2k-1},  s_{2 m_0},  s_{2 m_0}) }{\partial s_{2k-1}}  
> 0  \Leftrightarrow y  < \frac{2 s_2  s_3 \cdots
s_{2k-1}s_{2k-1}  s_{2 m_0}^2}{3}.$$
The inequality above is valid since $y < s_2,$ therefore,  we obtain that the function 
$H(s_2, s_3, \ldots , s_{2k-1}, s_{2 m_0}, s_{2 m_0}) $ is increasing in 
$s_{2k-1},$
whence 
\begin{multline*}
H(s_2, s_3, \ldots, s_{2k-2}, s_{2k-1}, s_{2 m_0}, s_{2 m_0})\geq   
H(s_2, s_3, \ldots, s_{2k-2}, s_{2 m_0}, s_{2 m_0}, s_{2 m_0})\\
=-\frac{y^{2k}}{s_2^{2k-1} s_3^{2k-2}\cdots s_{2k-2}^3 s_{2m_0}^3}  + \frac{y^{2k+1}}{s_2^{2k} s_3^{2k-1}\cdots s_{2k-2}^4 s_{2m_0}^6}.
\end{multline*}
Applying similar arguments we get the following chain of inequalities.
$$H(s_2, s_3, \ldots , s_{2k}, s_{2k+1}) 
 \geq   H(s_2, s_3, \ldots , s_{2k-1}, s_{2 m_0}, s_{2 m_0})\geq $$
$$H(s_2, s_3, \ldots , s_{2k-2}, s_{2 m_0}, s_{2 m_0}, s_{2 m_0}) \geq \ldots \geq
H(s_{2 m_0}, s_{2 m_0}, \ldots, s_{2 m_0}, s_{2 m_0}).$$
Thus, we have proved (\ref{mm6}).

We substitute the inequality (\ref{e4}) and  (\ref{mm6}) into  (\ref{e33}) to get the following
\begin{equation}
\label{mm8}\psi(y) \geq - \sum_{k=0}^{2m_0} \frac{(-1)^k y^k}{s_{2 m_0}^{\frac{k(k-1)}{2}}}=  
- S_{2m_0}(- \sqrt{s_{2 m_0}}y, g_{\sqrt{s_{2 m_0}}}),
\end{equation}
where  $g_a$ is a partial theta function and $S_{n}(x, g_a):=\sum _{j=0}^{n} x^j a^{-j^2}$ is its
partial sum.  By our assumption $ (\sqrt{s_{2 m_0}})^2 =s_{2 m_0} = q_{j-2m_0 +3}$ 
and $j \geq j_0 +2 m_0 -3,$  so $ s_{2 m_0} = q_{j-2m_0 +3} \geq q_{j_0} \geq c_{2m_0},$ 
and we conclude that  $S_{2m_0}(x, g_{s_{2 m_0}}) \in \mathcal{L-P}$ (see Theorem C). 
Whence, by part (4) of Theorem C,  there exists $x_0 \in (-  (\sqrt{s_{2 m_0}})^3, - \sqrt{s_{2 m_0}})$ such that 
$S_{2m_0}(x_0, g_{s_{2 m_0}}) \leq 0.$ We put $- \sqrt{s_{2 m_0}}y_0 := x_0,$
i.e. $y_0 := - \frac{x_0 }{\sqrt{s_{2 m_0}}} \in (1, s_{2 m_0}) \subset (1, s_2),$ and we have 
$$S_{2m_0}(- \sqrt{s_{2 m_0}}y_0, g_{\sqrt{s_{2 m_0}}}) \leq 0.$$
Substituting the last inequality in (\ref{mm8}) we obtain: 
\begin{equation}
\label{mm9}
\psi(y_0) \geq - S_{2m_0}(- \sqrt{s_{2 m_0}}y_0, g_{\sqrt{s_{2 m_0}}}) \geq 0.
\end{equation}
Using (\ref{mm9}) and substituting (\ref{mm8}) into (\ref{ee1}), we get:
\begin{align*}
(-1)^j \psi(x) > h(x) = \frac{x^{j+1}}{q_2^j q_3^{j-1} \cdots q_j^2 q_{j+1}} \cdot \psi(y_0) \geq 0,
\end{align*}
which is  the desired inequality.
It remains to recall that $x_j   := \frac{q_2 q_3 \ldots q_{j} q_{j +1}}{y_0}, $
and, since $y_0\in (1, s_2) = (1, q_{j+1}),$ we have $x_j \in (q_2 q_3 \ldots q_{j} , 
q_2 q_3 \ldots q_{j} q_{j +1}). $
\end{proof}

Now we apply the following lemma.

\begin{Lemma}
\label{th:lm5} (\cite[Lemma 2.1]{ngthv4}). {\it  Let $f(z) = \sum_{k=0}^\infty  a_k z^k,$ 
$a_k >0,  k=0, 1, 2, \ldots ,$ be an entire function 
such that $2\sqrt[3]{2} \leq  q_2(f) \leq q_3(f) \leq q_4(f) \leq \cdots.$
For an arbitrary integer $k \geq 2$ we define 
$$\rho_k(f) := q_2(f)q_3(f) \cdots  q_k(f) \sqrt{q_{k+1}(f)}.$$ 
Then, for all sufficiently large $k$,  the function $f$ has 
exactly $k$ zeros on the disk $\{z\  :\   |z| <  \rho_k(f) \}$
counting multiplicities. }
\end{Lemma}

Let us choose an arbitrary $k \geq 2$,  being large enough to get the
statement of the previous lemma, and $k \geq j_0 +2 m_0 - 2. $
Then the number of zeros of $\varphi$ (counting multiplicities) in the disk $\{z\  :\   |z| <  
q_2q_3\cdots  q_k \sqrt{q_{k+1}} \}$ is equal to $k.$ By Lemma  \ref{th:lm4}
we have 
$$\sign \varphi(x_{j_0 +2 m_0 -3})= - \sign   \varphi(x_{j_0 +2 m_0 -2}); \  
\sign   \varphi(x_{j_0 +2 m_0 -2})  $$ 
$$=
- \sign  \varphi(x_{j_0 +2 m_0 -1});   \ldots ;\   \sign  \varphi(x_{k-2}) = -\sign  \varphi(x_{k-1}), $$
and 
$$0 < x_{j_0 +2 m_0 -3} < x_{j_0 +2 m_0 -2} < \cdots < x_{k-1}  <
q_2q_3 \cdots  q_k < q_2q_3\cdots  q_k \sqrt{q_{k+1}}.$$ 
Hence, the
function $ \varphi$ has $k - j_0 -2m_0 +3 $ sign changes in the interval
$(0, q_2q_3\cdots  q_k \sqrt{q_{k+1}}),$  whence the number of real zeros
of $ \varphi$ in the disk $\{z\  :\   |z| <  q_2q_3\cdots  q_k \sqrt{q_{k+1}} \}$
is at least $k - j_0 -2m_0 + 2. $ Therefore, the number of nonreal zeros of $ \varphi$ in 
this disk is less than or equal to $j_0 +2m_0 -2.$ Since $k$ is an arbitrary large enough 
integer, we get that $ \varphi$  has not more than $j_0 +2m_0 -2$ nonreal zeros.

{\bf Aknowledgement.} The research was supported  by  the National Research Foundation of Ukraine funded by Ukrainian State budget in frames of project 2020.02/0096 ``Operators in infinite-dimensional spaces:  the interplay between geometry, algebra and topology''.

The authors would like to thank the reviewer for careful reading and valuable remarks.


\begin{thebibliography}{8}

\bibitem{Boh} 
A.~Bohdanov, Determining bounds on the balues of barameters for a bunction 
$\varphi_{a}(z, m) =\sum_{k=0}^\infty \frac{z^k}{a^{k^2}}(k!)^m,$ 
$m \in (0,1), $ to Belong to the Laguerre--P\'{o}lya Class, \textsl{ Comput. Methods Funct. Theory},
(2017), DOI:10.1007/s40315-017-0210-6. 

\bibitem{BohVish} 
A.~Bohdanov and A.~Vishnyakova, On the conditions for entire 
functions related to the partial theta-function to belong to the Laguerre--P\'{o}lya class, 
\textsl{J. Math. Anal. Appl.}, \textbf{434 }, No. 2, (2016),  1740--1752, 
DOI:10.1016/j.jmaa.2015.09.084

\bibitem{cc1}
T.~Craven and G.~Csordas,  Complex zero decreasing sequences,
\textsl{Methods Appl. Anal.},  \textbf{2} (1995), 420--441.

\bibitem{GR}
G.~Gasper and M.~Rahman, Basic Hypergeometric Series, \textsl{Encyclopedia 
of mathematics and its applications}, Cambridge University Press, United Kingdom, 
Cambridge, 2004.

\bibitem{HW}
I.~I.~Hirschman and D.~V.~Widder, {\em  The Convolution Transform},
Princeton University Press, Princeton, New Jersey, 1955.

\bibitem{hut}
J.~I.~Hutchinson,  On a remarkable class of entire functions,
\textsl{Trans. Amer. Math. Soc.},  \textbf{25} (1923), 325--332.

\bibitem{klv}
O.~Katkova, T.~Lobova and A.~Vishnyakova, On power series
having sections with only real zeros,  \textsl{ Comput. Methods Funct. Theory},  
\textbf{3}, No 2, (2003), 425--441.

\bibitem{klv1}
O.~Katkova, T.~Lobova and A.~Vishnyakova, On entire functions having 
Taylor sections with only real zeros,  \textsl{ J. Math. Phys., Anal., Geom.}, 
\textbf{11}, No. 4, (2004), 449--469. 

\bibitem{kos5} {V.~P.~Kostov}
{On a partial theta function and its spectrum} /
{V.~P.~Kostov } // Proceedings of the Royal Society of Edinburgh 
Section A: Mathematics. -- 2016. -- Vol. 146, N 3. -- P. 609-623.

\bibitem{kos5.0} {V.~P.~Kostov}
{The closest to 0 spectral number of the partial theta function} /
{V.~P.~Kostov} // C. R. Acad. Bulgare Sci. -- 2016. -- Vol. 69. -- P. 1105--1112.

\bibitem{kosshap}
V.~P.~Kostov, B.~Shapiro, 
Hardy-Petrovitch-Hutchinson's problem and 
partial theta function,  \textsl{Duke Math. J.}, \textbf{162}, No. 5 (2013), 
825--861. 

\bibitem{lev}
B.~Ja.~Levin, \textsl{Distribution of Zeros of Entire Functions},
Transl. Math. Mono., 5, Amer. Math. Soc., Providence, RI, 1964;
revised ed. 1980.

\bibitem{ngthv1}
T.~H.~Nguyen and A.~Vishnyakova, On the entire functions from the Laguerre--P\'olya 
class having the decreasing second quotients of Taylor coefficients, \textsl{Journal of 
Mathematical Analysis and Applications}, \textbf{465}, No. 1 (2018), 348 -- 359, \url{https://doi.org/10.1016/j.jmaa.2018.05.018}.

\bibitem{ngthv2} T.~H.~Nguyen and A.~Vishnyakova, On the necessary condition for 
entire function with the increasing second quotients of 
Taylor coefficients to belong to the Laguerre--P\'olya class, 
\textsl{Journal of Mathematical Analysis and Applications}, \textbf{480}, No. 2 (2019),
\url{https://doi.org/10.1016/j.jmaa.2019.123433}.

\bibitem{ngthv3} T.~H.~Nguyen and A.~Vishnyakova, On the closest to zero roots and the second quotients of 
Taylor coefficients of entire functions from the Laguerre--P\'olya I class, 
\textsl{Results in Mathematics}, \textbf{75}, No. 115 (2020), \url{https://doi.org/10.1007/s00025-020-01245-w}.

\bibitem{ngthv4} T.~H.~Nguyen and A.~Vishnyakova, On the entire functions from the Laguerre--P\'olya I class 
having the increasing  second quotients of Taylor coefficients, 
\textsl{Journal of Mathematical Analysis and Applications}, \textbf{498}, No. 1 (2021), 
\url{https://doi.org/10.1016/j.jmaa.2021.124955}.

\bibitem{ngth1} T.~H.~Nguyen, On the conditions for a special entire function related to 
the partial theta-function and the Euler function to belong to the Laguerre--P\'olya class, 
\textsl{Computational Methods and Function Theory} (2021), \url{https://doi.org/10.1007/s40315-021-00361-0}.

\bibitem{O}
N.~Obreschkov, Verteilung und Berechnung der Nullstellen reeller
Polynome, VEB Deutscher Verlag der Wissenschaften, Berlin, 1963.

\bibitem{iv}
I.~V.~Ostrovskii,  On zero distribution of sections and tails of
power series, \textsl{Israel Math. Conference Proceedings,}
\textbf{15} (2001), 297--310.

\bibitem{pol}
G.~P\'olya, Collected Papers, Vol. II Location of Zeros, (R.P.Boas
ed.) MIT Press, Cambridge, MA, 1974.

\bibitem{polsch}
G.~P\'olya and J.~Schur, \"Uber zwei Arten von Faktorenfolgen in der Theorie
der algebraischen Gleichungen, \textsl{ J. Reine Angew. Math.}, \textbf{144} (1914),
pp. 89--113.

\bibitem{sokal}
A.~D.~Sokal, The leading root of the partial theta function, \textsl{Advances in Mathematics}, \textbf{229}, No. 5 (2012), 2063 -- 2621.

\bibitem{warnaar} S.~O.~Warnaar, Partial theta functions, 
\url{ https://www.researchgate.net/publication/327791878_Partial_theta_functions}.

\bibitem{warnaar1} S.~O.~Warnaar, Partial Theta Functions. I. Beyond the Lost 
Notebook \textsl{Proc. London Math. Soc.}, \textbf{87}, No. 3 (2003), 363 -- 395.

\end{thebibliography}
\end{document}